\documentclass[reqno,11pt]{amsart}

\usepackage{graphicx,graphics,caption, subcaption, enumerate,color}

\usepackage{amsmath,amssymb,amsfonts,bm}
\usepackage{cases}

\usepackage{comment}

\usepackage{hyperref}
\hypersetup{
    bookmarks=true,         
    unicode=false,          
    pdftoolbar=true,        
    pdfmenubar=true,        
    pdffitwindow=false,     
    pdfstartview={FitH},    
    pdftitle={My title},    
    pdfauthor={Author},     
    pdfsubject={Subject},   
    pdfcreator={Creator},   
    pdfproducer={Producer}, 
    pdfkeywords={keyword1} {key2} {key3}, 
    pdfnewwindow=true,      
    colorlinks=false,       
    linkcolor=red,          
    citecolor=green,        
    filecolor=magenta,      
    urlcolor=cyan           
}

\usepackage{setspace}

 \evensidemargin=\oddsidemargin
\newdimen\dummy
\dummy=\oddsidemargin 
\newtheorem{theorem}{Theorem}
\newtheorem{corollary}[theorem]{Corollary}

\newtheorem{remark}{Remark}

\newcommand{\norm}[1]{\left\Vert#1\right\Vert}

\newcommand{\inpd}[2]{\left\langle #1, #2 \right\rangle}

\newcommand{\Real}{\mathbb {R}}

\newcommand{\eye}{\mbox{I}}

\newcommand{\tr}{\textsf{T}}

\newcommand{\Plane}{\mathcal{P}}

\newcommand{\mfd}{\mathcal{M}}
\newcommand{\ts}{\mathcal{T}}

\newcommand{\argmin}{ \operatornamewithlimits{argmin} }

\newcommand{\gad}{gentlest ascent dynamics }

\usepackage{lipsum}
\makeatletter
\g@addto@macro{\endabstract}{\@setabstract}
\newcommand{\authorfootnotes}{\renewcommand\thefootnote{\@fnsymbol\c@footnote}}%
\makeatother

\begin{document}
\title{ }

\begin{center}
  \LARGE 
An Iterative Minimization Formulation for Saddle-Point  Search \par \bigskip

  \normalsize
  \authorfootnotes
  Weiguo Gao\footnote{email: {\it wggao@fudan.edu.cn.} }\textsuperscript{1,2},  
 Jing Leng\textsuperscript{1},
  Xiang Zhou\footnote{email: {\it xiang.zhou@cityu.edu.hk}. Corresponding address: 
  Department of mathematics, City University of Hong Kong, Tat Chee Ave, Kowloon, Hong Kong SAR.
   }\textsuperscript{3} \par \bigskip

  \textsuperscript{1}School of Mathematical Sciences,  Fudan University,  Shanghai, 200433, China. \par
  \textsuperscript{2}MOE Key Laboratory of Computational Physical Sciences,  Fudan University,  Shanghai, 200433,  China. \par
  \textsuperscript{3}Department of Mathematics,
        City  University of Hong Kong,  \par Tat Chee Ave, Kowloon,  Hong Kong.\par \bigskip

  \today
\end{center}



\date{}

\maketitle

\date{}
\maketitle

\section*{Abstract}
This paper proposes  and analyzes  an  iterative minimization formulation for searching  index-1 saddle points of an energy function.
This formulation differs from other eigenvector-following methods by  constructing 
a new objective function near the guess at each iteration step.
This leads to a quadratic convergence rate,  
in comparison to the linear case of the gentlest ascent dynamics 
(E and Zhou, nonlinearity, vol 24, p1831, 2011) and many other existing methods.
We also propose  the generalization of the new methodology for saddle points of higher index and for constrained energy functions  on  manifold.

\

{{\bf Keywords}: saddle point, energy landscape, eigenvector-following, gentlest ascent dynamics, iterative minimization 
}
\

{{\bf  Mathematics Subject Classification (2010)}  Primary 65K05, Secondary  82B05	  }
\section{Introduction }

There have been considerable attentions for a long time to numerical methods of 
searching local minima of a continuous nonlinear  function. The widespread availability of the efficient optimization 
algorithms for large scale problems  has greatly assisted  the numerical studies of  theoretical  physics, chemistry  and biology.
In computational chemistry, for example,  it is of great interest to   look for metastable states of molecular configurations, which correspond to local minima of 
an energy function. 
Normally,  the traditional optimization procedures are very successful at locating a nearby metastable state.
 However, 
of more interest is the transition states in these molecular systems, 
which are the saddle points of the energy function.
When it comes to the location of the 
transitions states,  the minimization approach leaves a lot to be desired.

Transition states are characterized as stationary points having one, and only one, negative Hessian eigenvalues (e.g.  see \cite{Energylanscapes}).
This type of saddle points is usually named as  the index-1 saddle points.
There have already been a various   number of advanced  algorithms   during the past decades and
have proved more efficiently in searching saddles for many practical problems in chemistry and material sciences.
The  contributions include, but not limited to,  the following list: 
the activation-relaxation techniques\cite{ART1998}, the dimer method \cite{Dimer1999}, the nudged elastic band method\cite{NEB1998},  
and the  string method\cite{String2002, String2007,Ren2013}.
Interested readers can refer to  the reference \cite{Henkelman2000}.
The first two methods in the above list are the examples of ``single-state''  (or  ``surface walker'') algorithms  and the last two 
are examples of ``chain-of-state'' algorithms.
The ``single-state'' algorithms mainly follow the ``eigenvector-following'' methodology 
\cite{Crippen1971,cerjan1981,Energylanscapes} --- the system is moved uphill along the  eigenvector   (``min-mode'') corresponding to the smallest eigenvalue
of  the Hessian matrix .  Therefore,  these methods drive the system away from the local minimum 
and push it to some index-1 saddle point if the convergence is achieved. 
Numerous  applications to practical problems  have shown that  these ``eigenvector-following'' -type (or ``min-mode'') methods 
generally have a much larger attraction domain for convergence to index-1 saddle points  than the 
 root-finding methods. In addition,  the specificity of selecting index-1 saddles renders these methods more favorable 
than  the root-finding methods.  One more benefit of using ``eigenvector-following''  ideas 
over the Newton's root-finding method 
is that the explicit information of Hessian matrix is usually not required in numerical implementation.

Recently, there are rising mathematical interests in writing  the  ``eigenvector-following'' methodology  \cite{Crippen1971}  in 
  the form of a  dynamical system. For instance,  one of the authors 
  have proposed  the ``gentlest ascent dynamics''  (GAD)\cite{GAD2011,SamantaGAD},
  which  is  a coupled dynamical system of both  a position variable
 and a direction variable. 
 A different but similar  dynamical system  for  the dimer method\cite{Dimer1999}
 is further pursued  in \cite{DuSIAM2012}  by introducing one more dimer length variable.
 
 In GAD, the dynamics flow is defined on the product space of the position in the configuration space and  the direction 
 in its tangent space.
 The position variable 
 describes  the escape trajectory  from the basins of attraction of the local minima.
  The direction variable in GAD simultaneously evolves to try to follow the min-mode
 of the Hessian matrix, although it  does not have to be exactly the min-mode at any time.  It is  proved  that the stable equilibrium points  of  
this dynamical system are  the index-1 saddle points  of the energy function while 
the local minima of the energy function are turned into the index-1 saddle points of GAD.
This interesting property  invites one to attempt  to think of  GAD as a counterpart of 
the very basic  steepest descent dynamics (SDD), which   converges to a  local minimum as time goes to infinity.
 GAD and SDD are both the simplest  flow based on the gradient of the energy function in the configuration space
 and the convergence rates  are both linear for these two methods.

SDD is closely related the steepest descent method,  the simplest
gradient method for unconstrained optimization, which  can be traced back to Cauchy \cite{CauchySD}.
The  analogy between SDD and  GAD  is  a tentative attempt to compare 
the framework of various optimization algorithms  and that of  the saddle search algorithms. 
It is well known that the steepest descent method is ineffective for unconstrained optimization because of its slow convergence rate. Historically, 
many  better alternative optimization techniques have been developed  to achieve   superlinear or quadratic convergence rate,
for instance,   Newton's method, L-BFGS,    nonlinear conjugate gradient method,  and so on \cite{NOCEDAL1999}.
We are interested to  ask what could be  the possible analogues of these advanced optimization methods in the context of  
saddle search problem   and how to improve the linear convergence of  the GAD as well as   other popular saddle search algorithms.

Our motivation is  thus to follow the above mentality   and includes the following two-fold tasks.  Firstly,  we want to present a new mathematical framework with connection to  some optimization problem,
rather than in form of  a dynamical system,  with the hope that the GAD is a natural ``gradient flow'' of the associated optimization problem.  
Secondly,  the new formulation should be able to provide a super linear or quadratic convergence rate and carry  more  flexibility in designing 
faster algorithms. This paper  focuses on the first goal in  theoretical aspect and a partial discussion of the second goal with preliminary numerical experiments.
 The full discussion of developing faster numerical methods and  applying to real problems  
will be  presented  in a separate  article.

The formulation  we propose in this note is an iterative 
minimization scheme. At each iteration,  a new  objective function is constructed based on the given energy function
by using the information of  the current values of the position and the eigenvector of the
smallest eigenvalue.
Then a local minimizer  of this objective function is assigned to  
the new value of the position at the next iteration.  
This iterative scheme can be completely described by a continuously differentiable mapping.
We discover that the Jacobian matrix of this mapping vanishes at the saddle point and 
it follows that  the convergence rate  of the iterative minimization scheme is quadratic.

The authors notice that 
a few variant  techniques    have been proposed in efforts to improve  the  efficiency  of the ``single-state''-type algorithms for saddle search, such as \cite{HBK2005,KS2008, Cances2009,LOR2013}.
However,  all these methods either improve  the rotation step of solving the eigenvector or improve   the translation step of moving the position
in    {\it the} dimension along  the obtained direction.
The  resulting overall effect on the errors  actually still only has the linear convergence rate in the configuration space.


The rest of this paper is  organized as follows.  We first briefly review the gentlest ascent dynamics in Section 2. 
In Section 3, we formulate   our iterative minimization scheme
for index-1 saddles and analyze the convergence rate.   
We also discuss the situation with constraints for saddle points in Section 4.
The generalization for index-$m$ ($m>1$)  saddles are discussed in Section 5. 
Several  numerical  examples are presented in Section 6 to illustrate our theory.
Section 7 is the concluding remarks.

\section{Review of  gentlest ascent dynamics (GAD)}

The gentlest ascent dynamics 
is a mathematical model in form of the dynamical system to 
describe the escape of the basin of attractions in the gentlest way and the convergence to a saddle.
Given a smooth energy function  $V$ on the configuration space, say $\Real^d$,  the gentlest ascent dynamics  is the following dynamic system
defined on the phase space $\Real^d \times \Real^d$
\begin{subnumcases}{\label{GAD-grad}}
 \dot{x}  = - \nabla V(x) + 2\frac{\inpd{\nabla V(x)}{ v}}{\inpd{v}{v}} v,  \label{eqn:GADx} \\
 \gamma \dot{v}  = - \nabla^2 V(x)v + \frac{\inpd{v}{ \nabla^2 V(x) v}}{\inpd{v}{v}} v  .\label{eqn:GADv}
\end{subnumcases}
Here $\inpd{\cdot}{\cdot}$ is the dot product in the Euclidean space  $\Real^d$ and the relaxation constant $\gamma$ can be any positive real number.
 The second equation  (\ref{eqn:GADv}) attempts to find the direction that corresponds to the smallest eigenvalue of the Hessian matrix $\nabla^2V(x)$.
 The second term is to impose the normalization condition that $\norm{v}=\sqrt{\inpd{v}{v}}=1$.
 The last term in the first equation \eqref{eqn:GADx} reverses the component of the gradient force in  the direction $v$
 to drive the system uphill in the direction of $v$.

 It is shown in \cite{GAD2011} that the saddle point  of the original function $V$ is the stable equilibrium  point of the GAD.  
We  recall this result   in the following proposition for convenience. 

\bigskip
\noindent
{\bf Proposition.}
Assume that the energy function $V$ is $\mathcal{C}^4(\Real^{d}; \Real)$. 
\begin{enumerate}[(a)]
\item  If $(x_{*},v_{*})$ is an equilibrium point of the \gad (\ref{GAD-grad}) and  
$\norm{v_*}=1$, then 
$v_*$  is an eigenvector of $\nabla^2 V(x_{*})$ corresponding to one eigenvalue $\lambda_{*}$, 
and $x_{*}$ is an  equilibrium point of the steepest descent dynamics  of $V$, i.e., $\nabla V(x_{*})=0$.

\item Suppose that  $x_s$ is a stationary point of $V$, i.e., $\nabla V(x_s)=0$. 
Let $v_1,v_2, \cdots, v_d$ be the normalized eigenvectors of 
the Hessian $\nabla^2 V(x_s)$,  and the associated eigenvalues be
$\lambda_1,\lambda_2,\cdots, \lambda_d$, respectively.
Then 
for all $i=1,\cdots,d$,  $(x_s, v_i)$ is an equilibrium point of the \gad \eqref{GAD-grad}.
Furthermore, among these $d$  equilibrium points, there exists
one  pair $(x_s,v_{i'})$ which is linearly stable for  \gad \eqref{GAD-grad}, 
 if and only if $x_s$ is an index-$1$ saddle point  of the function $V$,
 or equivalently,  the eigenvalue $\lambda_{i'}$ corresponding to $v_{i'}$
 is the only negative eigenvalue of $\nabla^2 V(x_s)$.

\end{enumerate}

 \bigskip

When the GAD converges,
$v$ corresponds to the smallest  eigenvalue of the Hessian $H(x)\triangleq \nabla^2 V(x)$ at the saddle point
$x_{s}$. Actually, for any frozen  $x$ in the equation   \eqref{eqn:GADv}, 
the  steady state of  the solution $v(t)$ solves the following minimization problem for Rayleigh quotient
\begin{equation}
\label{eqn:minv}
\min_{\|u\|=1} {u}^\tr{H(x) u},
\end{equation}
and the equation \eqref{eqn:GADv} is just a steepest descent dynamics (rescaled in time by $\gamma$)  for the minimization problem \eqref{eqn:minv}.
In the limit of $\gamma\to 0$, $v(t)$ approaches the eigenvector of the smallest eigenvalue instantly, 
and GAD is reduced to traditional ``eigenvector-following'' methodology. 
In this case,  $v$  can be viewed   as a function  $v(x)$.
For finite $\gamma$, the equation \eqref{GAD-grad} couples the dynamics of $x$ and $v$ simultaneously
and still preserves the convergence to  saddle points.

In contrast to the flow for $v$, 
the dynamics equation \eqref{eqn:GADx} for the position $x$, however, is {\it not} in the form of
the steepest descent dynamics of some  scalar function.
 To see this, denoting   the  GAD force as $F(x)$:
\[F_{i}(x) \triangleq f_{i}(x) - 2\sum_{k}v_{i} f_{k}(x)v_{k} \]
where  $f(x)\triangleq -\nabla V(x)$. It is easy to see that 
$\frac{\partial F_{i}}{\partial x_{j}} = -H_{ij}+ 2 v_{i} \sum_{k} H_{kj}v_{k}=-H+2vv^\tr H$
while its transpose 
$\frac{\partial F_{j}}{\partial x_{i}} =-H+2Hvv^\tr.$
The necessary condition for the dynamics of $x$ being of a gradient type is that the Hessian $H$ commutes with the rank-$1$ matrix
$vv^\tr$, which  generally does not hold since $v$ may not be the exact  eigenvector of $H$ in GAD. 
Even in the $\gamma \to 0$ limit where $v=v(x)$ is indeed the eigenvector of $H(x)$,  the Jacobian matrix 
for $F_{i}(x) = f_{i}(x) - 2\sum_{k}v_{i}(x) f_{k}(x)v_{k}(x) $ is still not symmetric.
 
Therefore,  the GAD is not as simple as a steepest descent flow 
and no  underlying energy function to drive this dynamics.  
In  the next section, we shall show that  the GAD can be approximated by  a steepest descent flow of a new  
objective function which is  locally constructed. This leads to an iterative minimization formulation.

\section{The Iterative minimization formulation }

In this section, we discuss how to define a new objective function to drive the system
toward an index-1 saddle point of the original energy function. 
The intuitive idea   is to  change the sign of the energy function $V$ along some direction, rather than reversing  the
direction of the force as in  the GAD. The resulting Hessian then changes the sign of the smallest eigenvalue
while keeps the other eigenvalues  the same. 

\subsection{The iterative scheme}
\label{ssec:IS}

The framework we start with is  the following   iterative expression 
\begin{subnumcases}{\label{IMF}}
v^{(k+1)} =\underset{\norm{u}=1} {\argmin} \  u^{\tr}H(x^{(k)})u,\\
x^{(k+1)} =\underset{y}{\argmin} \ \left (  V(y) +  W^{(k)}(y)\right), 
\end{subnumcases}
where $W^{(k)}$ is an unknown function to be determined.  We need construct $W^{(k)}$ such  that $x^{(k)}$ 
 converges to a saddle point of $V$.

The following two choices of the function $W^{(k)}$  serve our purpose.
\begin{equation}
\label{eqn:Wk}
W_1^{(k)} (y) = W_1 (y; x^{(k)}, v^{(k+1)}) , 
\qquad
W_2^{(k)} (y) = W_2 (y; x^{(k)}, v^{(k+1)}) , 
\end{equation}
where with the abuse of the notation,  we  define 
\begin{align}
 W_1 (y; x, v)  &\triangleq  -2V(y)+ 2 V\left(y- v v^\tr (y-x)\right),  \label{eqn:W1}\\
 W_2 (y; x, v)  &\triangleq -2V\left(x+ vv^\tr (y-x)\right ).  \label{eqn:W2}
\end{align}
which are  two $\Real^d \to \Real$  functions parameterized by the position $x$ and the  normalized direction $v$.
Therefore,  the new objective function $V+W$ depends
on the current position $x$ and the direction $v$. In equation  \eqref{eqn:Wk}  for the choice of $W^{(k)}$,  the  direction $v^{(k+1)}$
is computed from the given $x^{(k)}$ via the Hessian matrix.
Therefore, the equation \eqref{IMF} is actually an iterative scheme mapping $x^{(k)}$ to $x^{(k+1)}$.

Given a position $x$ and a   direction $v$, 
we then have an  affine hyperplane, denoting as $\Plane_{x,v}$,    through $x$  with the normal $v$, i.e., 
$\Plane_{x,v} = \{y: (y-x)^{\tr}v=0\}$.  
Introduce the projection matrix  $\Pi_v$ and $\Pi_v^\perp=\eye-\Pi_v$, where
\[ \Pi_v u = vv^\tr u.\] 
Then,  the  argument in the second term of  $W_1$  is the point 
\[ y- v v^\tr (y-x) = x+ \Pi^\perp_v (y-x),\]
which  is  the projection of the point $y$ on the affine hyperplane   $\Plane_{x,v}$.
The position in $W_2$,  $x+vv^\tr(y-x)=x+\Pi_v (y-x)$, is the projection on the ray at $x$ with the direction $v$.
%

The intuition of  the definitions of $W_1$ and $W_2$ is the following.
If $y$ lies on the ray  along $v$, then $W_2=-2V(y)$.
Consequently, 
the  new energy  function
$V(y)+W_2(y; x,v)$ is to modify the potential $V(y)$ 
by reversing the sign of $V$  in the direction $v$.
For the choice of $V(y)+W_1(y; x,v)$, which is equal to $-V(y)+2V(x+\Pi^\perp(y-x))$,
can viewed as the reverse of the sign of $-V$ (instead of $V$) on the  $d-1$ dimensional affine plane $\Plane_{x,v}$.
Let us take a simple example of the  quadratic function $V(y)=\frac12 \sum_{i=1}^d \mu_i y_i^2 $
where $\mu_1<0<\mu_2<\cdots<\mu_d$.  The zero  vector is the index-1 saddle point of $V$.
$v=(1,0,\cdots,0)$ is  the eigenvector corresponding to the smallest eigenvalue $\mu_1$.
Then, 
\begin{align*}
 V(y)+W_1(y;x,v) &= \mu_1 x_1^2  - \frac12 \mu_1 y_1 ^2 + \frac12 \sum_{i=2}^d\mu_i y_i^2 ,
 \\
 V(y)+W_2(y;x,v) &= -\sum_{i=2}^d \mu_i x_i^2  - \frac12 \mu_1 y_1 ^2 + \frac12 \sum_{i=2}^d\mu_i y_i^2 .
\end{align*}
The difference of the above two functions  of $y$ is just a constant $2V(x)$.
Both of them are the convex quadratic functions of $y$ and they share the same Hessian matrix
$\mbox{diag}\{-\mu_1,\mu_2,\cdots,\mu_d\}$ and  the same minimizer  $0$, which is exactly the saddle point of $V$. 
So for any initial position $x^{(0)}$,  the next iteration $x^{(1)}$ 
is the true solution.
%

%
%
%

\subsection{Convergence result}

We have formulated the saddle search problem as a fixed point problem in   the iterative scheme   \eqref{IMF}
together with the defined $W_1$ and $W_2$ in \eqref{eqn:W1} and \eqref{eqn:W2}.
In fact,  the function $W^{(k)}$ in the iterative scheme \eqref{IMF} 
can be some linear combination of $W_1$ and $W_2$ to   achieve our purpose, too. In addition,
the constant $2$ showing in $W_1$ and $W_2$ can be relaxed.
In the next,  we shall consider this general case
to  define  the mapping from $x^{(k)}$ to $x^{(k+1)}$.
Denote this mapping for the iteration as $\Phi(x)$. 
We shall show that the Jacobian matrix  of $\Phi$ vanishes at the index-1 saddle point.
This implies the iterative  scheme is of quadratic convergence. 

\bigskip

\begin{theorem}
\label{thm:main}
Assume that $V(x)\in \mathcal{C}^3(\Real^d;\Real)$.
For each $x$, let  $v(x)$  be the normalized eigenvector corresponding to the smallest eigenvalue of the Hessian matrix $H(x)=\nabla^2 V(x)$, i.e., 
\[v(x)=\underset{u\in \Real^d, \, \| u\|=1}\argmin u^\tr H(x)u.\]
Given  two real numbers $\alpha$ and $\beta$ satisfying $\alpha+\beta > 1$,  
we define the following function of the variable $y$,
\begin{equation}
\label{eqn:new_obj}
\begin{split}
L(y ; x, \alpha,\beta) 
=  
 (1-\alpha)V(y) +\alpha V\bigg(y-v(x)v(x)^\tr(y-x)\bigg) 
          \\
          - \beta  V \bigg( x+v(x)v(x)^\tr(y-x)\bigg).
\end{split}
\end{equation}
Suppose that $x^*$ is an index-1 saddle point of the  function $V(x)$, i.e,
$\nabla V(x^*)$ has only one negative eigenvalue $\lambda(x^*)$.
Then the following statements are true.

\begin{enumerate}
\item[(i)] 
   $x^*$ is  a local  minimizer of $L(y ; x^*, \alpha,\beta) $.
\item[(ii)]
There exists a   neighborhood $\mathcal{U}$ of $x^*$ such that for any  $x\in \mathcal{U}$,    
$L(y; x,\alpha,\beta)$ is strictly convex  in $y \in    \mathcal{U}$
and thus has a unique minimum in $\mathcal{U}$.
\item[(iii)]  
Define the mapping    $\Phi :   x\in \mathcal{U} \to \Phi(x) \in  \mathcal{U} $ 
 to be the unique local minimizer  of $L$ in $\mathcal{U}$ for any $x\in \mathcal{U}$. 
 Further assume that  $\mathcal{U}$ contains   no other stationary point of $V$ except $x^*$.
   Then 
 the mapping 
  $\Phi$ has only one  fixed point $x^*$ .
   
  \item[(iv)] $\Phi(x)$  is differentiable  in $\mathcal{U}$.
  The derivative of $\Phi$ vanishes at $x^*$, i.e.,  the Jacobi matrix 
\begin{equation}
 \Phi_x (x^*)=0.
\end{equation}

\end{enumerate}
 \end{theorem}
\bigskip 
\begin{proof}
\
{\it Part (i)}:

\
 We calculate the first and second  derivative  of $L(y; x,\alpha,\beta)$ with respect to $y$.
The first order derivative is 
\begin{equation}
\label{eqn:devL}
\begin{split}
 \nabla_y L =  (1-\alpha) \nabla V(y) + \alpha 
(I-vv^\tr)\nabla V \left( y -vv^\tr  (y-x)\right ) 
\\-\beta vv^\tr \nabla V\left(x+vv^\tr(y-x) \right).
\end{split}
\end{equation}
So, it is clear that $\nabla_y L(x^*; x^*, \alpha,\beta)=0$ for any constants $\alpha$ and $\beta$ since $\nabla V(x^*)=0$. 

The Hessian matrix of $L$ is 
\begin{equation}
\label{eqn:HessL}
\begin{split}
  \nabla^2_y L(y;x,\alpha,\beta)
=&(1-\alpha)\nabla^2 V(y) + \alpha (I - vv^\tr)
  \nabla^2 
V(y-vv^\tr(y-x))  (I - vv^\tr) \\
 &\qquad  \qquad\qquad - \beta   vv^\tr  \nabla^2 
V(x+vv^\tr(y-x))  vv^\tr,
\end{split}
\end{equation} 
which is simplified at $y=x^*$ and $x=x^*$ as follows 
\[
\begin{split}
  \nabla^2_y L(x^*;x^*,\alpha,\beta)
 = & H(x^*) -  (\alpha+\beta)\lambda(x^*) v(x^*) v(x^*)^\tr .
\end{split}
\] 
where the fact $H(x^*)v(x^*)=\lambda(x^*)v(x^*)$ is applied.
Since $\lambda(x^*)<0$, then $ \nabla^2_y L(x^*;x^*,\alpha,\beta)
$ is  positive definite if $\alpha+\beta>1$.

\vskip 1em

{ \it Part (ii)}

The assumption that
$ V\in\mathcal{C}^3$ and $x^*$ is the index-1 saddle of $V$ implies that
the eigendirection $v(x)$ is continuously differentiable at $x^*$.
Then  Equation  \eqref{eqn:HessL}, together with the continuity of  $\nabla^2V(x)$ and $v(x)$ at $x^*$, implies that 
the Hessian, $\nabla^2_y L(y;x,v(x))$,
which is treated now  as  a function of two variables $y$ and $x$,  is continuously
differentiable at $(y,x)=(x^*,x^*)$. 
In {\it Part (i)}, we proved that at the point $(x^*,x^*)$, the Hessian is   positive-definite as $\alpha+\beta>1$.
Thus, there exists a neighborhood  of $(x^*,x^*)$, denoted as $\mathcal{N}$, such that the  Hessian  $\nabla^2_y L(y;x)$
at each $(y,x)\in \mathcal{N}$ is still  positive-definite.
Select a neighorhood  in the form of  product of two convex sets  $\mathcal{U}\times \mathcal{U}$ inside the $2d$-dim set $\mathcal{N}$, 
then the set $\mathcal{U}$ is the desired one.

\vskip 1em

{ \it Part (iii)}:
 Suppose that there is  a second fixed  point $\hat x \in \mathcal U$ such that  $\Phi(\hat{x})=\hat{x}$. Then
$ \nabla_y L(\hat{x} ; \hat{x}) =0$.
From the equation \eqref{eqn:devL},  \[ 
 \nabla_y L(\hat{x} ; \hat{x}) =    \nabla V(\hat{x}) - (\alpha+\beta) v(\hat x)v(\hat x)^\tr \nabla V\left(\hat{x} \right)=0.
\]
 Since $\alpha+\beta \neq 1$, then $\nabla V(\hat x)=0$. 
But 
there is only one stationary point $x^*$ in $\mathcal{U}$, so  $\hat x$ has to be  the point $x^*$. 
\vskip 1em

{\it Part (iv)}:
\

For each $x_0\in \mathcal{U}$, $(\Phi(x_0),x_0)$ is the solution of the first order equation  
$\nabla_y  L(\Phi(x_0);x_0)=0$. It is clear that 
$\nabla_y L(y;x)$ is continuously differentiable at all  $(y,x)$ in   $\mathcal{U}\times \mathcal{U}$.
In addition,  $\nabla_y^2 L(y;x)$  is strictly  positive-definite from   ${\it Part (ii)}$, thus non degenerate  in  $\mathcal{U}\times \mathcal{U}$. Therefore, the Implicit Function Theorem implies that 
$\Phi(x)$ is Lipschitz continuous and differentiable near $x_0$.

Next, we calculate the derivative of the mapping $\Phi$,  denoted as  $\Phi_x(x)$.
For each $x\in \mathcal{U}$,   $\Phi(x)$ is a solution of the first order equation 
\eqref{eqn:devL}. Thus the following equation   holds    
\begin{equation}
\label{eqn:fixedpt1}
\begin{split}
(1-\alpha)\nabla V\left(\Phi(x)\right)&+\alpha\left(I-v(x)v(x)^\tr \right)\nabla 
V\left(\varphi_1(x)\right)
- \beta\,  v(x)v(x)^\tr  \nabla  V \left(\varphi_2(x)\right)=0.
\end{split}
\end{equation}
where 
\[ \varphi_1(x)= \Phi(x)-v(x)v(x)^\tr(\Phi(x)-x), \qquad  \varphi_2(x)=x+v(x)v(x)^\tr (\Phi(x)-x).\]
%
Now we take the derivative  with respect to $x$  on  both sides of  \eqref{eqn:fixedpt1}, then
\begin{equation}
\label{eqn:fixedptdx}
\begin{split}
(1-\alpha) H(\Phi) \Phi_x &+ \alpha \left(I-vv^\tr\right) H(\varphi_1) \varphi_{1,x} - 
\alpha v^\tr \nabla V(\varphi_1) J - \alpha   v  \nabla V(\varphi_1)^\tr J 
\\
=&~\beta  vv^\tr H(\varphi_2) \varphi_{2,x}
+ \beta v^\tr \nabla V(\varphi_2) J 
 + \beta v \nabla V(\varphi_2)^\tr J .
\end{split}
\end{equation}
where the derivatives $J=\frac{\partial v(x)}{\partial x}$ and
$\Phi_x=\frac{\partial \Phi}{\partial x}$  are the Jacobi matrix of $v(x)$ and $\Phi(x)$ respectively.
The derivatives $\varphi_{1,x},\, \varphi_{2,x}$ are defined similarly.
 
Note that $\Phi(x^*)=x^*$ thus $\varphi_1(x^*)=\varphi_2(x^*)=x^*$.
Consequently $\nabla V(\varphi_1)$ and $\nabla V(\varphi_2)$ both vanish at $x^*$ since $\nabla V(x^*)=0$.
Meanwhile, since \[
\begin{split}  \varphi_{1,x} &=\Phi_{x}
- vv^\tr \left( \Phi_{ x} -I\right)
-  v^\tr (\Phi-x)J - v (\Phi-x)^\tr J,
\\
\varphi_{2,x} &= I
+ vv^\tr \left( \Phi_ x -I\right)
+  v^\tr (\Phi-x)J + v(\Phi-x)^\tr J,
\end{split}
\]
then in particular at $x=x^*$, we have
\[
\begin{split}
 \varphi_{1,x} (x^*)   =
&
\left(I - v(x^*)v(x^*)^\tr \right) \ \Phi_{x}(x^*) 
+  v(x^*)v(x^*)^\tr ,  \\
  \varphi_{2, x} (x^*) = &
  I-v(x^*)v(x^*)^\tr
+ v(x)v(x)^\tr \Phi_x (x^*) .
\end{split}
\]

Therefore,  by noting $\nabla V(x^*)=0$ again, the equation \eqref{eqn:fixedptdx}    at $x=x^*$   becomes  
\begin{equation}
\label{eqn:fixedptdx2}
\begin{split}
&\quad  (1-\alpha) H(x^*)\Phi_x(x^*)\\
&+ \alpha  \left(I-v(x^*)v(x^*)^\tr\right)
H(x^*)\left(I-v(x^*)v(x^*)^\tr\right)
\Phi_{x}(x^*) \\
&+\alpha\left(I-v(x^*)v(x^*)^\tr\right)\nabla^2V(x^*)v(x^*)v(x^*)^\tr\\
&-
~~~~~\beta v(x^*)v(x^*)^\tr H(x^*)  \left(I-v(x^*)v(x^*)^\tr\right)
\\
&-
\beta v(x^*)v(x^*)^\tr H(x^*)   v(x^*)v(x^*)^\tr \Phi_x(x^*)
\\
&=0.
\end{split}
\end{equation}
As $v(x)$ is the  
eigenvector of the Hessian $\nabla^2V(x)$, i.e., $H(x)v(x)=\lambda(x)v(x)$, 
 it is easy to verify that for each $x$, 
$
\left(I-v(x)v(x)^\tr\right)\nabla^2V(x)v(x)v(x)^\tr=0$ holds. Thus, any term  in \eqref{eqn:fixedptdx2}  without the Jacobi $\Phi_x$ vanishes. So,
the equation  \eqref{eqn:fixedptdx2}  gives the following linear equation 
\begin{equation}
\label{eqn:fixedptdx*}
  \bigg( H(x^*)- (\alpha+\beta)  \lambda(x^*)v(x^*)v(x^*)^\tr\bigg)
\Phi_{x}(x^*) =0
\end{equation}
which implies that $\Phi_x(x^*) =0$ if and only if $\alpha+\beta\neq 1$.

\end{proof}

The above theorem implies the  important property of the proposed iterative minimizing formulation
in \S\ref{ssec:IS} if the energy function $V$ has  a higher regularity $ \mathcal{C}^4$.

 \begin{corollary}
 \label{cor}
  Assume that $V(x)\in \mathcal{C}^4(\Real^d;\Real)$.
The  iterative scheme  $x^{(k+1)} = \Phi(x^{(k)})$ has exactly the second order  (local) convergence rate.
 \end{corollary}
\begin{proof}  
Since $V$ is $\mathcal{C}^4$, then $\nabla_y L$  has the regularity $\mathcal{C}^2$ and 
 in the neighborhood $\mathcal{U}$.
It follows  that  $\Phi(x)$ is continuously differentiable at $x^*$ 
based on the second order  pseudo expansion (\cite{BonnansSIAM})  for 
the first order equation $\nabla_y L=0$.

Since  $\Phi_x(x^*)=0$, then there is a neighborhood of $x^*$,
such that $\|\Phi_x(x)\|$ is strictly less than one in this neighborhood. 
Thus, the local convergence  comes from the contraction mapping principle. 

The second order convergence rate is due to the Jacobi matrix $\Phi_x(x^*)$ vanishes.
One  can carry out further  calculation and exam   that the second derivative of $\Phi(x)$ at $x^*$  does not  trivially vanish. 
So, the iteration $x \to \Phi(x)$  locally converges to $x^*$ exactly at quadratic rate.
\end{proof}

For the quadratic example in \S\ref{ssec:IS}, we have the following trivial result.
\begin{corollary}
If $V(x)=\frac12 x^\tr H x$ where $H$ is a constant symmetric matrix
and has only one negative eigenvalue.
Suppose that $\alpha,\beta$  in Theorem  \ref{thm:main}
satisfy  $\alpha+\beta>1$. 
Then  $ \Phi(x) =0 $ for all $x$.

\end{corollary}

We remark that  the $\Phi(x)$ is well-defined   in a local neighborhood of $x^*$. 
In implementations, the local solution of the new objective function $L(y;x^{(k)})$ in \eqref{eqn:new_obj} 
is pursued with the  initial guess   $y_0=x^{(k)}$. This choice of the initial guess is not only
very simple to pick up  but also excludes the possibilities of overshooting to    other local solutions which are not relevant to the saddle of interest.

%

\subsection{Solve subproblem of minimization}
\label{ssec:subsolver}
The iterative minimization formulation \eqref{IMF}  consists of solving a subproblem of minimization 
at each iteration.   Corollary \ref{cor}  suggests that the quadratic convergence rate is achieved  when the subproblem is solved 
exactly and the correct local minimizer is found.
In practice, one may not need to solve the subproblem of the minimization exactly or with high accuracy
and the superlinear convergence may be achieved in certain circumstances.
Many existing ``eigenvector-following" methods like  dimer method could be viewed as some
special discretisation for the subproblem.

We first present a result about the connection of the gentlest ascent dynamics and the iterative minimization formulation.

\begin{theorem}
\label{thm:IMF-GAD}
Assume $x^{(k)}$ is near the index-1 saddle point  $x^*$.
Suppose that one solves the subproblem $\min_y L(y;x^{(k)},\alpha,\beta)$
in Theorem \ref{thm:main} by only one single steep descent method with the step size 
$\delta t$,
\begin{equation}
x^{(k+1)} = x^{(k)} - \delta t\,  \nabla_y L(x^{(k)}; x^{(k)}, \alpha, \beta). 
\end{equation}
Then the sequence $\{x^{(k)}\}$ is the discrete solution of the   Euler method with the time step $\delta t$  for the  following version of the gentlest ascent dynamics 
\begin{equation}
\dot{x} = - \nabla V(x) + (\alpha + \beta) \Pi_{v(x)} \nabla V(x).
\end{equation} 
\end{theorem}
\begin{proof}
The conclusion is obvious by noting the following  and  the fact in equation  \eqref{eqn:devL},   
\[
\begin{split}
x_{k+1} =& x_k - \delta t\,  \nabla_y L(x_k; x_k, \alpha, \beta)  \\
&= x_k - \delta t \big( (1-\alpha) \nabla V(x_k) + \alpha (I - v(x_k) v(x_k) ^\tr \nabla V(x_k)) \\
&  \qquad \qquad  - \beta v(x_k) v(x_k)^\tr \nabla V(x_k)
\big).
\end{split}
\]
\end{proof}

\begin{remark} If the subproblem for the direction $v$ is also solved by the steepest descent method, 
then $(x^{(k)},v^{(k)})$  corresponds to the original version of GAD \eqref{GAD-grad}.
\end{remark}

The subproblem at each iteration consists of  the minimization for the position and the
 direction.  Some fast algorithms have been developed for solving the smallest eigenvector 
 problem, in particular where the Hessian is not explicit available and the force is calculated from 
 the first principle \cite{LOR2013} .
 The new numerical challenge to implement the iterative minimization formulation
 efficiently is the minimization of $L$ for the position to get new $x^{(k+1)}$.
 Of course, one is not limited to use the steepest descent method to solve this subproblem 
 as in  Theorem \eqref{thm:IMF-GAD}.
For example,  CG can be applied with certain level of tolerance. Details about these accelerating techniques 
will be postponed in a separate paper.

We now discuss the choice of two parameters  $\alpha$ and $\beta$ in our formulation. Theoretically by Theorem \ref{thm:main},  the condition that $\alpha+\beta>1$ is sufficient for the algorithm to  achieve the local
quadratic convergence. In practice,  a better choice of $\alpha$ and $\beta$
may help reduce the condition number of  the subproblem, 
 which is ratio of the maximum eigenvalue and the smallest eigenvalue
 of  the Hessian $\nabla^2_yL$. 
The calculation in the proof of Theorem \ref{thm:main} has shown that 
at the saddle point $x^*$, the eigenvalues of $\nabla^2_y L$ are 
$(1-\alpha-\beta)\lambda_1,\lambda_2,\lambda_3,\cdots,\lambda_d$
where $\lambda_1<0<\lambda_2<\cdots<\lambda_d$ are eigenvalues of 
$\nabla^2 V(x^*)$.  Hence, to minimize the condition number of $\nabla^2_y L$,   the optimal choice of $\alpha$ and $\beta$ 
needs to satisfy  
\[     1+ \frac{\lambda_2}{|\lambda_1|} \leq \alpha+\beta \leq  1+ \frac{\lambda_d}{|\lambda_1|}, \]
 and the resulting  optimal condition number is  $\lambda_d/ \lambda_2$.
In practice,  a rough estimate of  $\lambda_2$ may be used to select the parameter $\alpha+\beta$ at each iteration.

When the initial guess of the iterative method is in the convex region of the original energy function, 
for example, a local minimum,  the function $L$ will have no lower bound locally
and the  minimization subproblem does not have a solution.
One can handle this situation using the traditional techniques implemented in  many eigenvector-following-type methods.
One simple remedy is to only seek the solution within a ball  or hypercube with a proper size around the current solution $x^{(k)}$. Such remedies are  not necessarily needed when $\lambda_1$ is negative.

\section{Saddles on  manifold }
\label{sec:smfd}
In some  applications, the configuration of the system may be subject to one or more constraints,
for instance  the conservation laws of some physical quantities.
Suppose that  these constraints specify a   Riemann manifold $\mfd$ embedded in $\Real^d$.
The index-1 saddle point of the energy function restricted on $\mfd$
is still the transition state of interest.  The calculation of the saddle point on the manifold 
calls for the attention for the effect of the constraints.  
In this section, we want to extend the  GAD and iterative minimization  scheme onto  the manifold $\mfd$.
This goal can be easily achieved for GAD  by a simple projection procedure, as in  \cite{DuCMS2009,DuJCP2012}. 
But it  requires extra work for the iterative minimization formulation.

We assume that 
the   manifold $\mfd$ is characterized by $p$ (independently) equality constraints: $c_i(x)=0$ for $i=1,2,\cdots, p$
where $c_i$ are   $\Real^d\to \Real$ smooth functions.
To maintain  a right mix   of abstraction and concreteness, we use the extrinsic variables $x$   in $\Real^d$ for $\mfd$.
The tangent space $\ts_x$   at each point $x$ of the manifold $\mfd$
is thus the   orthogonal complement  in  $\Real^d$   to the normal space spanned by  the gradients of the $p$ constraints,  $span\{\nabla c_i(x),i=1,2,\cdots,p\}$.
The concepts of the local minimum and the index-1 saddle point of the smooth energy function $V(x)$ can be  extended to the manifold case without any difficulty
\cite{optmfd_book}. We skip the rigorous math definitions since they are quite intuitive. 

We start with  the calculation of the eigenvector $v$ in the tangent space $\ts_x$ corresponding to the smallest eigenvalue
of the (projected) Hessian matrix of the energy function $V$.
This   direction $v$  minimizes  the Rayleigh quotient among all possible vectors in $\ts_x$:
\[
v =\underset{\norm{u}=1,  u\in \ts_x} {\argmin}  u^{\tr} \nabla^2 V(x)u,
\]
or equivalently 
\begin{equation}
\label{eqn:vc}
v =\underset{\norm{u}=1, u\in \Real^d} {\argmin} \bigg \{ u^{\tr} \nabla^2 V(x)u \bigg \vert  \inpd{\nabla c_i (x)}{ u}=0, \forall i=1,2,\cdots,p \bigg\}.
\end{equation}
The steepest descent flow of this constrained minimization problem is 
\[
\gamma \dot{v} = - \Pi_{\ts_x} \big[ \nabla^2 V(x) v  \big]+ \eta v
\]
where $\Pi_{\ts_x}$ is the orthogonal projection of $\Real^d$ to the vector space $\ts_x$
and  the scalar $\eta=\inpd{\Pi_{\ts_x} \big[ \nabla^2 V(x) v  \big]}{v}$ is to enforce the unit length of $v$.
Many existing fast algorithms for the original rotation step to solve $v$ in $\Real^d$ can be 
readily modified for the constrained problem \eqref{eqn:vc}.

Next, we discuss the dynamics or the iterations for the position variable $x$.
For the dynamics of $x$ in GAD, we can simply project the GAD force $(-\eye+2 vv^\tr) \nabla V(x)$  onto the  tangent  space $\ts_x$,
i.e., 
\[
\dot{x} = \Pi_{\ts_x} \big[  (-\eye+2 vv^\tr) \nabla V(x)  \big],
\]
then  the trajectory of GAD  stays on the manifold $\mfd$ all the time. 
However, for  the iterative minimization formulation \eqref{IMF},
the need of projection on $\mfd$  complicates our discussion.  
Specifically, for a given $x\in \mfd$ and $v\in \ts_x$, one can find  a geodesic 
curve on $\mfd$ by following the geodesic flow which 
can be described  in terms of these constraints functions $c_i(x)$ \cite{Botsaris1981295}.
Let $\xi(s)$ ($s\in \Real$) be the geodesic curve  satisfying $\xi(0)=x$ and $\xi'(0)=v$.
For each point $y\in \mfd$ near $x$, under some mild condition,
 we can define the projection of $y$ onto the 
 geodesic $\xi$ as $\xi(s_y)$ where 
$s_y\triangleq \argmin_s \mbox{dist}(\xi(s), y) $. Here  ``$\mbox{dist}$'' is the distance between two points of the manifold $\mfd$:
infimum of the lengths of all   continuously differentiable curve on $\mfd$ joining these two points.
The argument of  the $W_2$ function
is then the point which has minimal   distance on $\mfd$ to the curve 
$\xi(s)$, i.e., the ``projection'' of $y$ to $\xi$.  
Therefore, the counterpart of  $W_2$ on  $\mfd$
is 
\begin{equation}\label{eqn:W2g}
W_2(y)=-2V(\xi(s_y)). 
\end{equation}
In principle,   the same strategy can be applied for the 
$W_1$ function where the minimal distance to the set of geodesic curves whose tangents  are in $\ts_x$ 
but orthogonal to $v$ should be pursued .

%
%
In a nutshell, the iterative minimization scheme on the manifold $\mfd$ specified by the $p$ constraints $c_i(x)=0$ can be written as follows
\begin{subnumcases}{\label{IMF-c}}
x^{(k+1)} =\underset{ y}{\argmin} \ \left\{  V(y) +  W^{(k)}(y) \big \vert \,  c_i(y)=0, \forall i = 1, 2, \cdots,p\right\}, \\
v^{(k+1)} =\underset{\norm{u}=1, u\in \Real^d} {\argmin} \left \{ u^{\tr} \nabla^2 V(x^{(k+1)})u \big \vert  \inpd{\nabla c_i (x^{(k)})}{ u}=0, \forall i \right \},
\end{subnumcases}
where $W^{(k)}$ (which depends on $x^{(k)}$ and $v^{(k+1)}$) is defined through the above mentioned $W_1$,  $W_2$  or their linear combination  in the same way as in Theorem \ref{thm:main}.
To illustrate the above idea, the example of the sphere $S^2$ in $\Real^3$ is presented in Section \ref{ssec:s2}.
The numerical result for a quadratic energy function
on this manifold shows that  the iterative scheme \eqref{IMF-c} also has the quadratic convergence rate.

\section{Saddle with  higher index }
The reference \cite{GAD2011}  of  GAD has extended from the index-1 saddle point to saddle points of index-$m$ for  $m>1$ case 
with the help of dilation technique. 
Our new     iterative minimization  formulation proposed above for the index-1 saddle can also be  extended to the case of saddles with index more than $1$.
Suppose that one has found $m$ eigenvectors of the Hessian matrix $\nabla^2 V(x)$, 
 $v_1,v_2, \cdots, v_m$, 
corresponding the $m$ smallest  eigenvalues, respectively. We denote $S$ as the set of all subsets
of $\{1,\ldots,m\}$ except the empty set. For every $s\in S$, we have $s=\{i_1,\ldots,i_k\}\subset\{1,\ldots,m\}$ with
$1\le i_1<\cdots<i_k\le m$ and $k\le m$. The projection onto the plane spanned by $k$ column vectors $\{v_{i_1},v_{i_2}, \ldots,v_{i_k}\}$
is associated with the following matrix
\[ \Pi_s = V_s V_s^\tr = \left[ \begin{array}{cccc}
v_{i_1} & v_{i_2} & \cdots & v_{i_k}\end{array}\right]
\left[ \begin{array}{c}
v_{i_1}^\tr \\ v_{i_2}^\tr \\ \vdots \\ v_{i_k}^\tr\end{array}\right].
\]
Let $\Pi^\perp_s=\eye - \Pi_s$. 
The objective function for the subproblem, which is a generalization to the equation \eqref{eqn:new_obj}, is now given by
\begin{equation}\label{eqn:hi}
\begin{split}
L(y;x,\alpha,\beta)=\left(1-\sum_{s\in S}\alpha_s\right)V(y)+\sum_{s\in S}\alpha_sV(x+\Pi^\perp_s(y-x))\\
-\sum_{s\in S}\beta_sV(x+\Pi_s(y-x)).
\end{split}
\end{equation}
where $\alpha=(\alpha_s)_{s\in S}$ and $\beta=(\beta_s)_{s\in S}$.
For example, the  function \eqref{eqn:hi} in the index-2 case is 
\begin{align*}
&L(y;x,\alpha,\beta)=(1-\alpha_1-\alpha_2-\alpha_{1,2})V(y)\\
&\qquad +\alpha_1V\left(x+\Pi_1^\perp(y-x)\right)+\alpha_2V\left(x+\Pi_2^\perp(y-x)\right)+\alpha_{1,2}V\left(x+\Pi_{1,2}^\perp(y-x)\right)\\
&\qquad -\beta_1V\left(x+\Pi_1(y-x)\right)-\beta_2V\left(x+\Pi_2(y-x)\right)-\beta_{1,2}V\left(x+\Pi_{1,2}(y-x)\right).
\end{align*}
In parallel to Theorem \ref{thm:main}, we have the following theorem for the index-$m$ case.
Its  proof is   similar to the proof  of Theorem \ref{thm:main} but technically  lengthier  and thus is skipped. 
\begin{theorem}
Assume that $V(x)\in \mathcal{C}^4(\Real^d;\Real)$.
For each $x$, let $v_1(x),\ldots, v_m(x)$ be $m$ normalized eigenvectors corresponding to the smallest eigenvalues of the Hessian
matrix $H(x)=\nabla^2 V(x)$, i.e., 
\[
[v_1(x),\ldots,v_m(x)]=\underset{U=[u_1,\ldots,u_m],\,U^\tr U=I}\argmin\mathrm{trace}\ U^\tr \nabla^2V(x)U.
\]
The function $L(y;x,\alpha,\beta)$ of the variable $y$ is defined as in \eqref{eqn:hi}
and it is assumed that \[
\sum_{s\in S}\left(\alpha_s+\beta_s\right)>1,
\]

Suppose that $x^*$ is an index-$m$ saddle point of the  function $V(x)$.
Then the following statements are true.
\begin{enumerate}
\item[(i)] 
$x^*$ is a local minimizer of $L(y;x^*,\alpha,\beta)$.
\item[(ii)] There exists a neighborhood $\mathcal{U}$ of $x^*$ such that for any  $x\in \mathcal{U}$,    
$L(y;x,\alpha,\beta)$ has a unique minimum in $\mathcal{U}$.
We define $\Phi(x)$ to be this  minimizer  for the given $x$.  
\item[(iii)]  
 The mapping 
  $\Phi$ has only one unique fixed point $x^*$ in $\mathcal{U}$.
  \item[(iv)] The mapping  $\Phi $ is differential in $\mathcal{U}$. 
   The Jacobi matrix  of $\Phi$ vanishes at $x^*$, i.e., 
\[
 \Phi_x (x^*)=0.
\]

\end{enumerate}
\end{theorem}

As a consequence of the above theorem,  the following iterative scheme
\[
\begin{cases}
x^{(k+1)}=\underset{y}\argmin \  L(y;x^{(k)},\alpha,\beta)\\
[v^{(k+1)}_1,\ldots,v^{(k+1)}_m]=\underset{U=[u_1,\ldots,u_m],\,U^\tr U=I}\argmin\mathrm{trace}\ U^\tr\nabla^2V(x^{(k)})U
\end{cases}
\]
converges to the index-$m$ saddle point $x^*$ quadratically if the starting point $x^{(0)}$ is close enough to $x^*$.

\section{Examples}
\label{sec:example}

\subsection{A simple two dimensional  example}
Firstly, we review a two dimensional  example in \cite{GAD2011}.
$$
V(x,y)=\frac{1}{4}(x^2-1)^2+\frac{1}{2}\mu y^2
$$
where $\mu $ is  a positive parameter. 
For this system, $x_\pm =(\pm 1,0)$ are two local minima  and $(0,0)$ is the index-1 saddle point. The eigenvalues and eigenvectors of the Hessian matrix at a point $(x,y)$ are
\begin{align*}
  \lambda_1 &= 3x^2-1 \ \mbox{ and } \ v_1=(1,0),\\
  \lambda_2 &= \mu \ \mbox{ and } \ v_2= (0,1).
\end{align*}
Note that when $|x|\leq \sqrt{\frac{1+\mu}{3}} $, $\lambda_1\leq 0 <\lambda_2$.
The eigenvector corresponding to the smallest eigenvalue 
 is $v_1$ if $|x|<\sqrt{\frac{1+\mu}{3}} $ and becomes $v_2$ if $|x|>\sqrt{\frac{1+\mu}{3}} $.

Suppose that at iteration $k$, the position is $(x_k,y_k)$.
Then,  the modified objective functions  $V+W_1$  and $V+W_2$ in the iterative minimization  method is defined as follows
\begin{equation*}
\begin{cases} 
  V+W_1^{(k)}= -\frac{1}{4}(x^2-1)^2+\frac{1}{2}\mu y^2+\frac{1}{2}(x_k^2-1)^2  
 & \text{if $|x|<\sqrt{\frac{1+\mu}{3}}$}, 
\\
  V+W_1^{(k)}= \frac{1}{4}(x^2-1)^2-\frac{1}{2}\mu y^2+\mu y_k^2  
   &\text{if $|x|>\sqrt{\frac{1+\mu}{3}}$.}
\end{cases}
\end{equation*}

\begin{equation*}
\begin{cases} 
  V+W_2^{(k)}= -\frac{1}{4}(x^2-1)^2+\frac{1}{2}\mu y^2- \mu y_k^2   & \text{if $|x| \leq \sqrt{\frac{1+\mu}{3}}$}, 
\\
  V+W_2^{(k)}= \frac{1}{4}(x^2-1)^2-\frac{1}{2}\mu y^2-   \frac{1}{2}(x_k^2-1)^2  
   &\text{if $|x|>\sqrt{\frac{1+\mu}{3}}$.}
\end{cases}
\end{equation*}
These are piecewise continuous function and  the difference of $W_1$ and $W_2$ is only  a constant. 
In the domain where $|x|<\min \left (1, \sqrt{\frac{1+\mu}{3}} \right)$,  the original saddle $(0,0)$ is the unique  interior  minimal point.
Outside of this domain, the modified function $V+W_1$ or $V+W_2$ has no lower bound. 
 So,  the iterative minimizing method  works only when  the initial guess satisfies $|x|<\sqrt{\frac{1+\mu}{3}}$.

\subsection{The three-hole example}
In the second example, we study a two dimensional energy function in \cite{mfpt2003,tpt2006}
 where  there are three local minima.
The formula of this energy function is 
\[ 
   V(x,y)=3e^{-x^2-(y-\frac{1}{3})^2}-3e^{-x^2-(y-\frac{5}{3})^2}-5e^{(x-1)^2-y^2}
   -5e^{(x+1)^2-y^2}+0.2x^4+0.2\left(y-\frac{1}{3}\right)^4.
\]
Refer to Figure  \ref{pic:three-hole} for the contour plot.
\begin{figure}[htbp]
  \centering
  \includegraphics[width=.76\textwidth]{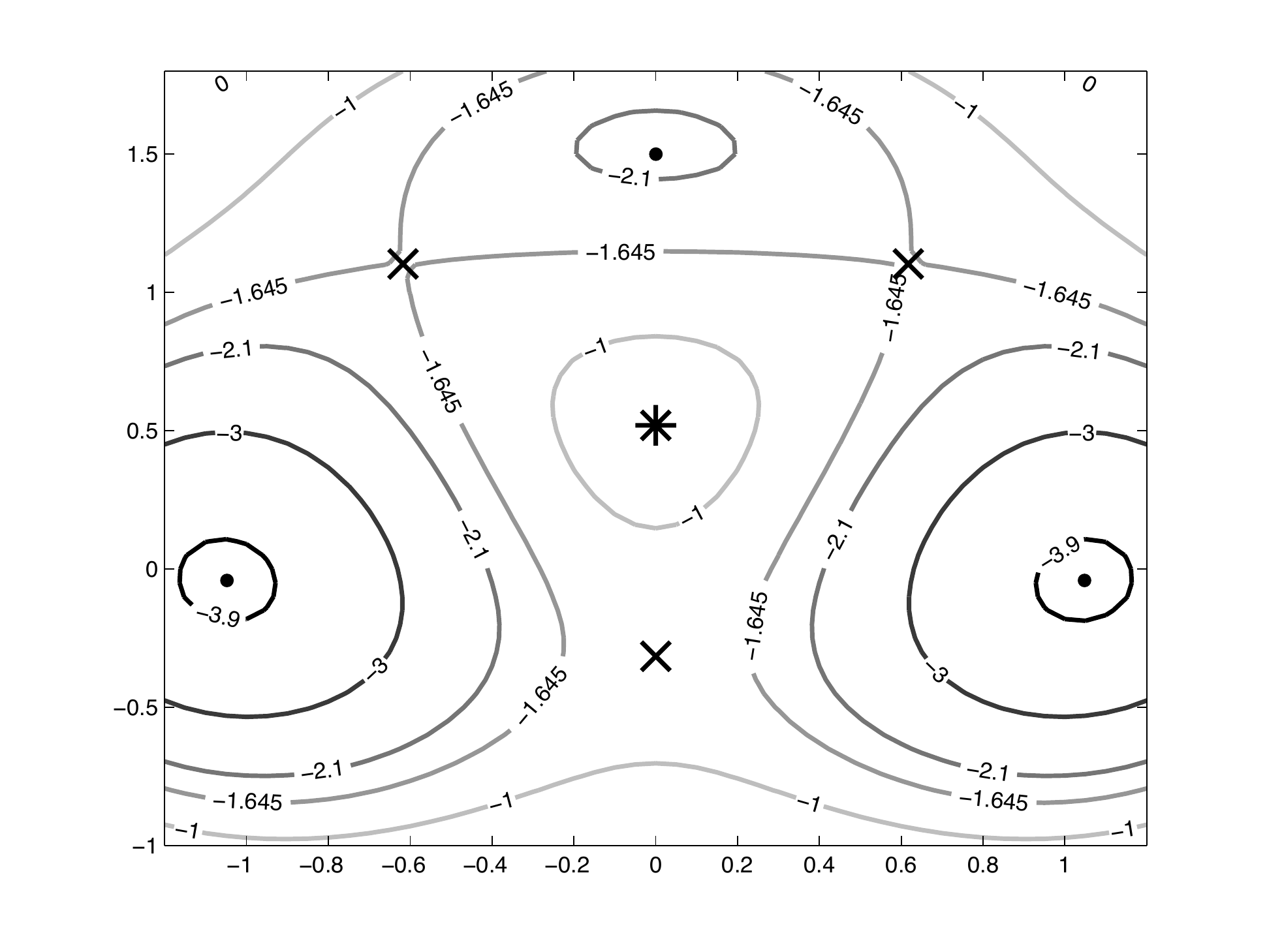}
  \caption{Three-hole potential: three minima (black dots) approximately at $(\pm 1,0)$ and $ (0,1.5)$, a maximum (``$*$'') at $(0,0.5)$ and three saddle points 
  (``$\times$'') at $(0,-0.31582)$ and $(\pm 0.61727,1.10273)$. }\label{pic:three-hole}
\end{figure}

Let $\mathrm{SP1}=(0,-0.31582)$ and $\mathrm{SP2}=(-0.61727,1.10273)$ be the two saddles of interests. 
We first demonstrate the quadratic convergence when the initial guess is near the saddle points.
 Table \ref{tab:hole-1} shows  the  errors at each iterations, from which  it is observed 
that the iterative scheme has the quadratic convergence rate. 
\begin{table}[htbp]
\begin{center}
\begin{tabular}{|c|c|c|c||c|c|c|}
  \hline
 Iter  & $(2,0)$      & $(0,2)$      & $(1,1)$                & $(2,0)$      & $(0,2)$      & $(1,1)$ \\\hline
 $1$   & 5.042e-002   & 2.979e-002   & 2.924e-002             & 1.672e-002   & 3.024e-002   & 4.342e-002 \\
 $2$   & 1.376e-005   & 5.470e-004   & 1.671e-004             & 9.327e-006   & 3.445e-004   & 3.194e-004 \\
 $3$   & 7.245e-011   & 2.573e-008   & 2.434e-008             & 2.527e-011   & 1.116e-008   & 1.233e-008 \\
 $4$   & 5.023e-016   & 5.551e-016   & 3.951e-016             & 2.482e-016   & 3.886e-016   & 4.965e-016 \\\hline
\end{tabular}
\end{center}
\caption{Errors of 6  runs with  random initial guesses circled at the target saddle point with 0.2 radius.
 Different values of  $(\alpha,\beta)$  for the modified  
objective functions in subproblem  are shown in the brackets.The left 3 runs converge to SP1 and the right 3 runs converge to SP2.}
\label{tab:hole-1}
\end{table}

If the initial guess is close to the local minima of $V$, then  the Hessian of $V$  at  the initial point
is positive-definite, while the modified objective function has one negative eigenvalue
and has no lower bound. As discussed in \S\ref{ssec:subsolver},  we  set a maximum step size $0.25$ both in $x$ and $y$ direction at each iteration
 to maintain the stability at this initial stage.
 Equivalently, it is to   solve the subproblem in the square box of size $0.25$.
 Table \ref{tab:hole-2} shows the result for initial points which are $0.1$ away from $(-1,0)$,
 one of the two deep minima. Some runs converge to SP2 and others converge to SP1.
It is observed that at the first few steps,  the decreasing of the errors is slow but 
when it approaches to the saddle, the smallest eigenvalue of the Hessian becomes negative,
and it follows that  the iterative minimization method starts to show quadratic convergence.
\begin{table}
\begin{tabular}{|c|c|c|c||c|c|c|}\hline 
 Iter  & $(2,0)$      & $(0,2)$      & $(1,1)$ & $(2,0)$      & $(0,2)$      & $(1,1)$ \\\hline
 $1$   & 8.434e-001   & 8.568e-001   & 8.496e-001             & 1.160e+000   & 1.170e+000   & 1.262e+000 \\
 $2$   & 6.891e-001   & 7.026e-001   & 6.954e-001             & 9.922e-001   & 9.970e-001   & 1.077e+000 \\
 $3$   & 5.731e-001   & 5.858e-001   & 5.790e-001             & 8.512e-001   & 8.537e-001   & 9.464e-001 \\
 $4$   & 5.216e-001   & 5.317e-001   & 5.263e-001             & 6.983e-001   & 6.939e-001   & 8.020e-001 \\
 $5$   & 3.862e-001   & 3.848e-001   & 3.841e-001             & 5.273e-001   & 5.028e-001   & 6.397e-001 \\
 $6$   & 1.776e-001   & 1.740e-001   & 1.742e-001             & 3.391e-001   & 3.030e-001   & 4.542e-001 \\
 $7$   & 1.983e-002   & 4.266e-002   & 1.987e-002             & 1.511e-001   & 1.291e-001   & 2.569e-001 \\
 $8$   & 7.314e-007   & 3.343e-004   & 2.834e-004             & 2.975e-002   & 1.207e-002   & 8.031e-002 \\
 $9$   & 3.756e-012   & 9.654e-009   & 4.088e-008             & 2.093e-006   & 2.879e-005   & 3.919e-003 \\
 $10$  &              & 6.810e-016   & 7.773e-015             & 1.734e-015   & 1.912e-011   & 7.345e-006 \\
 $11$  &              &              &                        &              & 3.140e-016   & 2.745e-011 \\\hline
\end{tabular}
\caption{Errors of 6 random runs with initial guesses circled at a local minimum  $(-1,0)$ with $0.1$ radius.
The  three runs  on the left  columns converge to SP1 and the  other three runs  on the right columns converge to SP2 respectively.}
\label{tab:hole-2}
\end{table}

We also  test the effects of the inexact solution of the subproblem. In solving the subproblem by 
the conjugate gradient  method, we  only perform three steps of 
conjugate gradient search.  The initial guesses
are chosen on the circle centered at the saddle points with $0.2$ radius. 
The results are shown in Table \ref{tab:hole-3} with two different parameter sets $(\alpha=2, \beta=0$ and $\alpha=0,\beta=2)$.
In comparison to the case of exact solution for subproblem in Table \ref{tab:hole-1}, 
the efficiency of the algorithm  is  not affected much 
and the local convergence rate is still quite close to the second order.

\begin{table}
\begin{tabular}{|c|c|c||c|c|c|c|}\hline 
 Iter  & $(2,0)$      &$(0,2)$       & $(2,0)$      & $(0,2)$ \\\hline
 $1$   & 4.476e-02   & 2.723e-02   & 5.096e-02   & 1.877e-02  \\
 $2$   & 1.262e-04   & 3.256e-04   & 7.756e-05   & 8.386e-04  \\
 $3$   & 2.863e-08   & 6.047e-06   & 1.270e-09   & 4.176e-05 \\
 $4$   & 5.317e-13   & 3.316e-10   & 7.830e-13   & 1.8827e-09 \\
 $5$   &              & 7.325e-13   &              & 4.3853e-11  \\\hline
\end{tabular}
\caption{Errors of 4  runs with random initial guesses circled at the target saddle points with $0.2$ radius. Use three-step nonlinear CG method to solve the subproblem of minimization inexactly. 
The left 2 runs converge to SP1 and the right 2 runs converge to SP2 respectively.
}
\label{tab:hole-3}
\end{table}

In the end, for this 2-D example,  we plot the domain of attraction for our  algorithm to compare with the performance of  the  Newton method.
Note that our purpose here is to look for index-1 saddle points.  We choose initial guesses from   $50\times 50$ grid points uniformly
 in the rectangular region $[-1.5, 1.5]\times[-1.5,2.0]$. These points  are labelled in Figure \ref{pic:da} by three different  marks in three colours,  according to  which saddle point (shown in cross sign) they converge to. 
 The grid point  is left blank in case of no convergence to any  saddle.  
The figure demonstrates that our IMF scheme has a larger (and continuous) domain of attraction 
for each saddle point. 
\begin{figure} [thbp]
\centering
\begin{subfigure}[b]{0.49\textwidth}
                \centering
                \includegraphics[width=\textwidth]{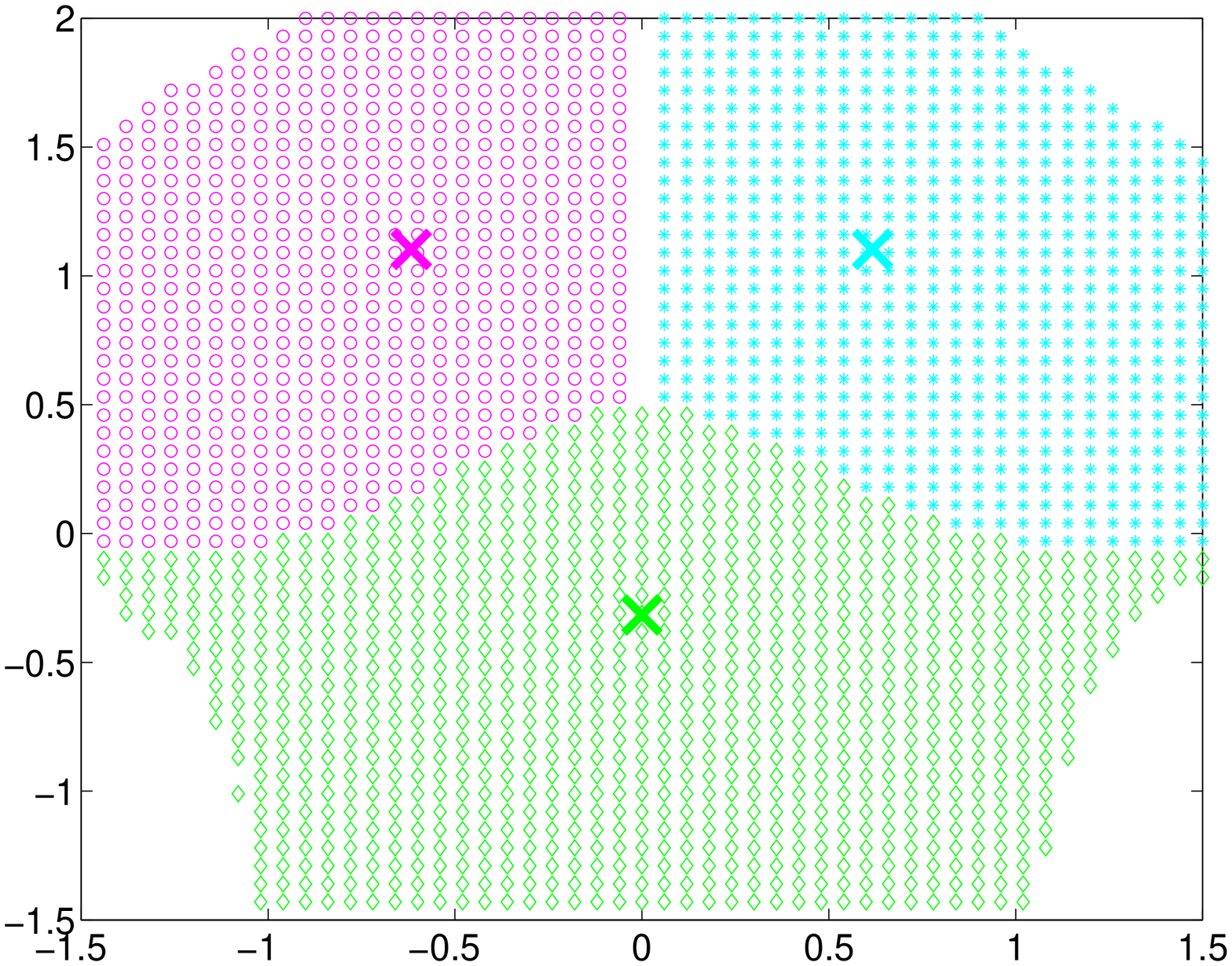}
                \caption{IMF}
                \label{fig:min1}
        \end{subfigure}%
        \hfill 
        \begin{subfigure}[b]{0.49\textwidth}
                \centering
                \includegraphics[width=\textwidth]{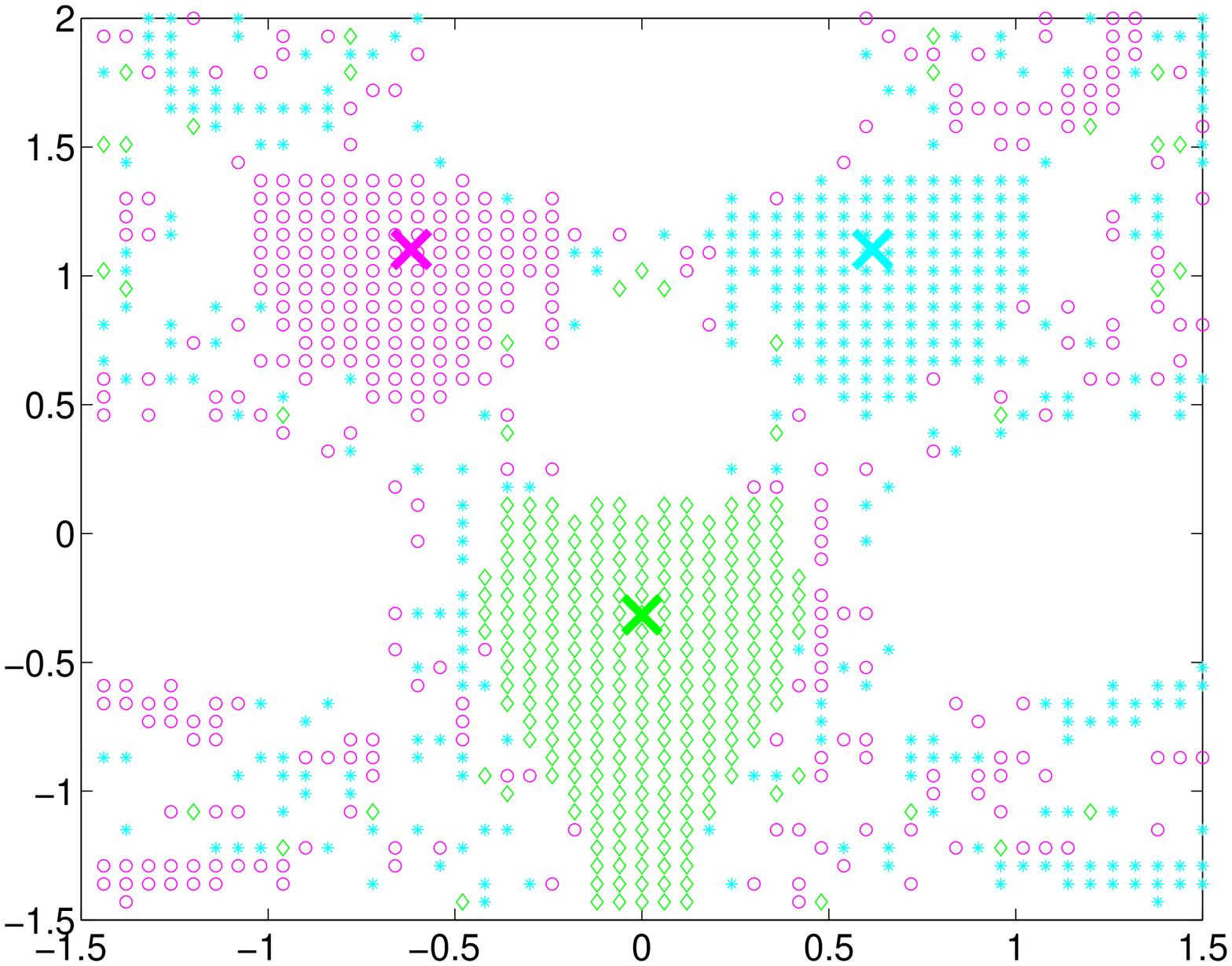}
                \caption{Newton}
                \label{fig:sp1}
        \end{subfigure}%
    \caption{ Comparison of domain of attractions for saddle search problem in our scheme and Newton scheme.}\label{pic:da}
\end{figure}

\subsection{A quadratic function on the sphere  $S^2$}
\label{ssec:s2}
We illustrate the proposal in Section \ref{sec:smfd} for the constrained  problem by considering 
a simple example of  $\mfd=S^2$ embedded in $\Real^3$ on which  a quadratic energy function $V(x_1,x_2,x_3)=x_1^2 + 2 x_2^2 + 3x_3^2$ is defined. 
The constraint  is   that $x_1^2+x_2^2+x_3^2=1$. It is easy to verify that  saddle points  $(0,\pm 1, 0)$,
 minimizers $(\pm 1, 0,0)$ and  maximizers $(0,0,\pm 1)$  of $V$ are  generated due to this constraint.

As mentioned in  Section \ref{sec:smfd},  for a given $x\in S^2$ and $v\in \ts_x(S^2)$,  the projection of a point $y$ on 
$S^2$ is associated with a geodesic curve $\xi(s)$. Here the geodesic $\xi$ is simply the great circle passing the point $x$ along the direction $v$.
Thus $\xi$ can be written in the parameterized form $\xi(\theta)=x \cos\theta+v\sin\theta$.
It follows that the geodesic distance 
$\mbox{dist} (\xi(\theta),y)=\arccos\inpd{\xi(\theta)}{y}$  
achieves the minimum at   $\theta=\theta_y$, where $\theta_y $ is equal to $ \arctan\frac{\inpd{v}{y} }{\inpd{x}{y}} $ or 
$\arctan\frac{\inpd{v}{y} }{\inpd{x}{y}}+ \pi $ depending on which value gives smaller distance.
 Then the projection point of $y$ is
$ 
x\cos\theta_y+v\sin\theta_y.
$
So we have   $W_2(y) = -2V( x\cos\theta_y+v\sin\theta_y)$ for this $S^2$ example.

Next we also derive $W_1$ expression for this $S^2$ case.
Since the tangent space $\ts_x(S^2)$ is two dimensional, 
the orthogonal complement of $v$ in this space is spanned by just a single vector, denoted as $\tilde{v}$.
It follows then that $W_1(y)=-2V(y)+2V(x\cos\tilde{\theta}_y+v\sin\tilde{\theta}_y)$   where 
$ \tilde{\theta}_y$ is defined likely as  $\theta_y$ by substituting $v$ by $\tilde{v}$.

The numerical results based on the construction of the above $W_1$ and $W_2$ are
presented in  Table \ref{tab:s2}. The initial guess is  $0.1$ distant  to  the minimum point $(1,0,0)$.
The numerical data of the errors between the solution and the true saddle point in this table again confirm  the quadratic convergence rate.
We remark that it is important to use  the projection associated with the geodesic curve
in the construction of   $W_1$ and $W_2$.
One alternative idea might  use the projection  in  the Euclidean space $\Real^3$ as if no constraints, then pullback to $S^2$.
For instance,  one may use  the following  $ W_2 (y) = -2V\left( R_x( v^\tr (y-x) v)\right ),$
where $R_x(u)=\frac{x+u}{\|x+u\|}$ is a  retraction 
mapping  the tangent  space $\ts_x$ to  the sphere $S^2$.  However,  our numerical result for the 
same example here 
shows that this choice  
gives only a linear convergence rate.  The  missing curvature information of the manifold in this naive orthogonal projection approach 
seems to be the reason of lowering the convergence order.

\begin{table}[htbp]
\begin{tabular}{|c|c|c|c|c|c|}
\hline
Iter & $1$& $2$& $3$& $4$  & $5$ \\
   \hline
$V+W_1$ & 1.3900e+00 & 
   2.3217e-01 &
  1.4234e-03 &
  2.0994e-08&
 1.7684e-15
   \\
$V+W_2$ & 1.2902e+00 & 
   5.6506e-01 &
   6.5923e-02 &
   8.7484e-06&
   1.2433e-16
   \\
   \hline
\end{tabular}
\caption{Errors of  $S^2$ example}
\label{tab:s2}
\end{table}

\subsection{An atomic model system}
This is an application of our method to the celebrated test problem of a $7$ atoms island on the (111) surface of an FCC metal \cite{Henkelman2000}. 
In this example 
the structure has a $6$-layer slab, each of which contains 56 atoms, and 7 atoms at the top of the slabs.  The bottom three layers in the slab are frozen. 
 There are $56\times 3+7=175$ atoms are free to move.
All the atoms in this simulation are identical. 

The interaction between the atoms is the simple pairwise additive Morse potential
\[
V(R)=A[e^{-2a(R-R_0)}-2e^{-a(R-R_0)}]
\]
with parameters chosen to reproduce diffusion barriers on Pt surfaces
($A=0.7102 \text{eV}$, $ a=1.6047{\AA}^{-1}$, $R_0=2.8970 {\AA}$).
This potential is cut and shifted by V($R_C$) where $R_C = 9.5 {\AA}$~is the cut-off distance. The minimum energy lattice constant $2.74412 {\AA}$ is used.

We show two local minima in Figure \ref{pic:175} as well as two saddle points.
All saddle points lead from the close packed heptamer (shown in red) to some adjacent state. 
\begin{figure} [thbp]
\centering
\begin{subfigure}[b]{0.45\textwidth}
                \centering
                \includegraphics[width=\textwidth]{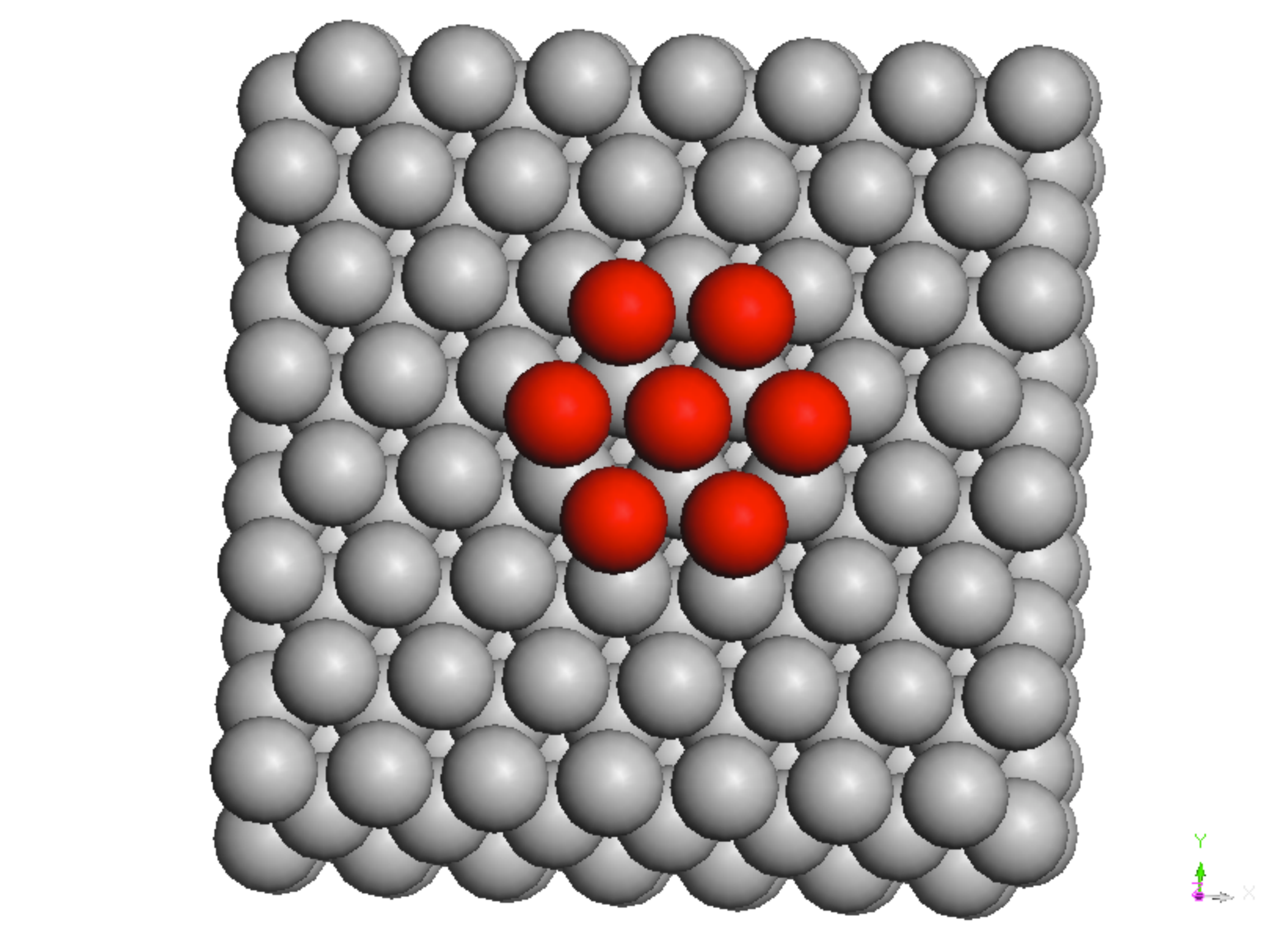}
                \caption{MIN1}
                \label{fig:min1}
        \end{subfigure}%
        \hfill 
        \begin{subfigure}[b]{0.45\textwidth}
                \centering
                \includegraphics[width=\textwidth]{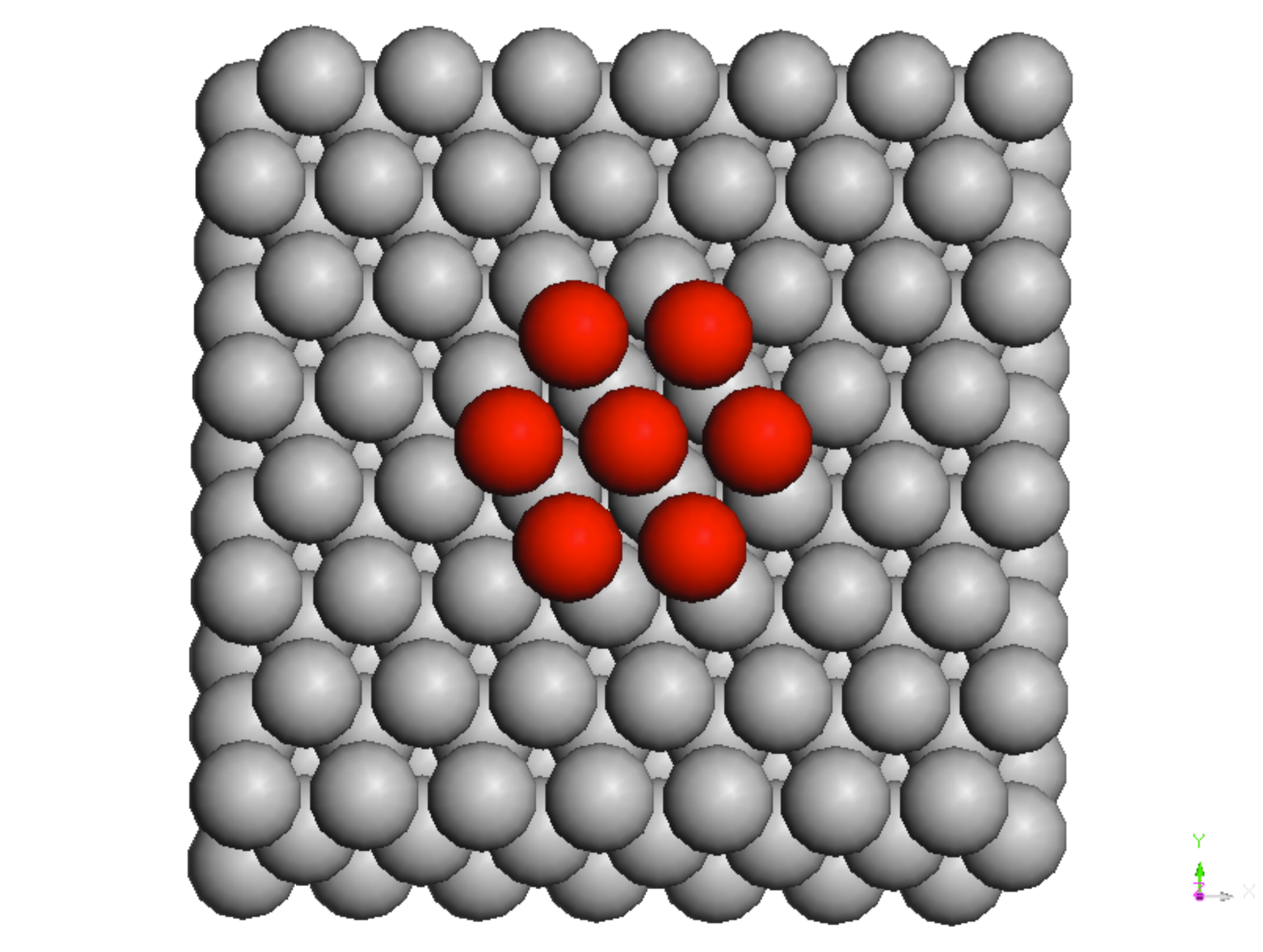}
                \caption{SP1}
                \label{fig:sp1}
        \end{subfigure}%
        \\
        
\begin{subfigure}[b]{0.45\textwidth}
                \centering
                \includegraphics[width=\textwidth]{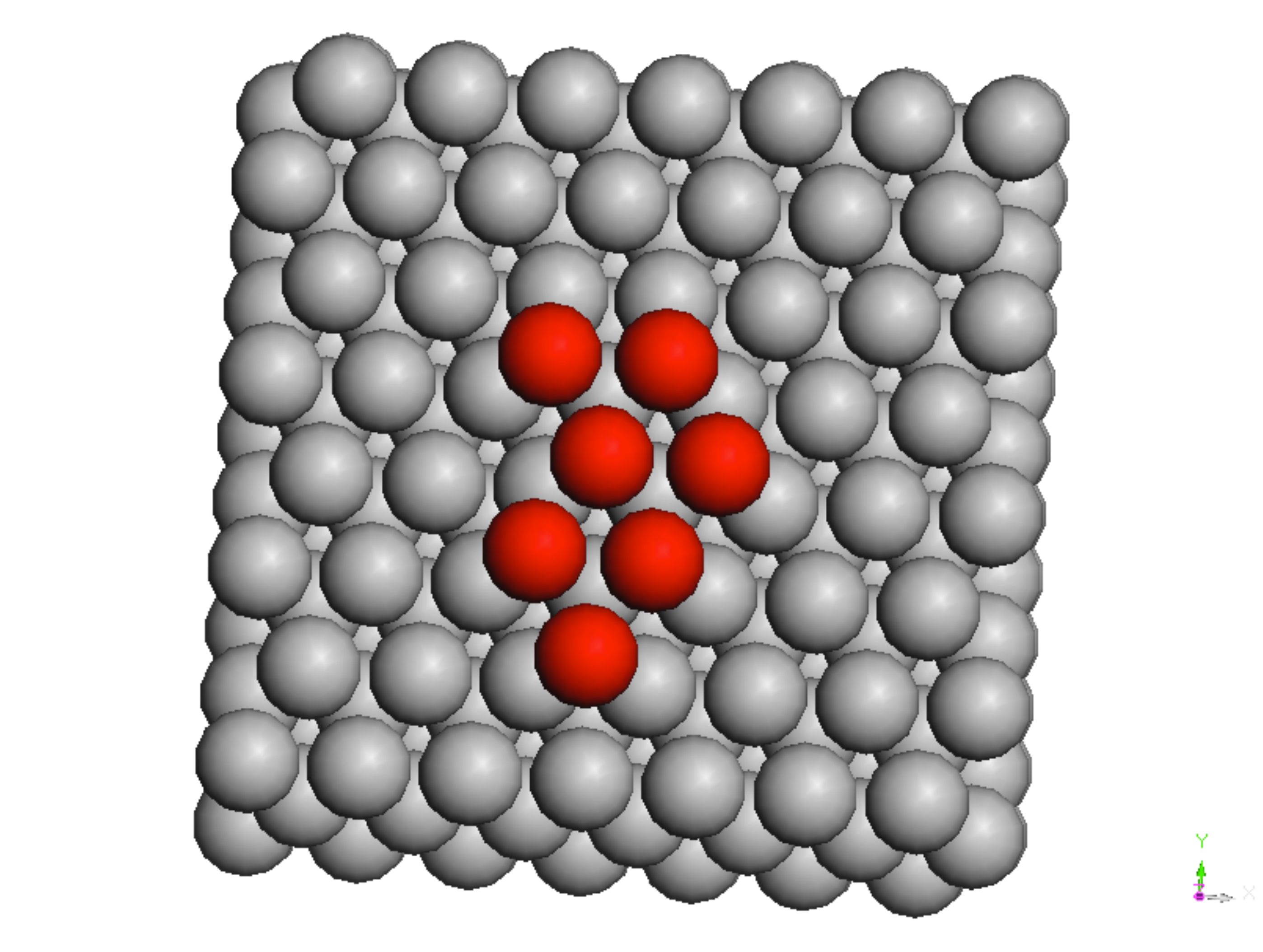}
                \caption{MIN2}
                \label{fig:min}
        \end{subfigure}%
        \hfill
        \begin{subfigure}[b]{0.45\textwidth}
                \centering
                \includegraphics[width=\textwidth]{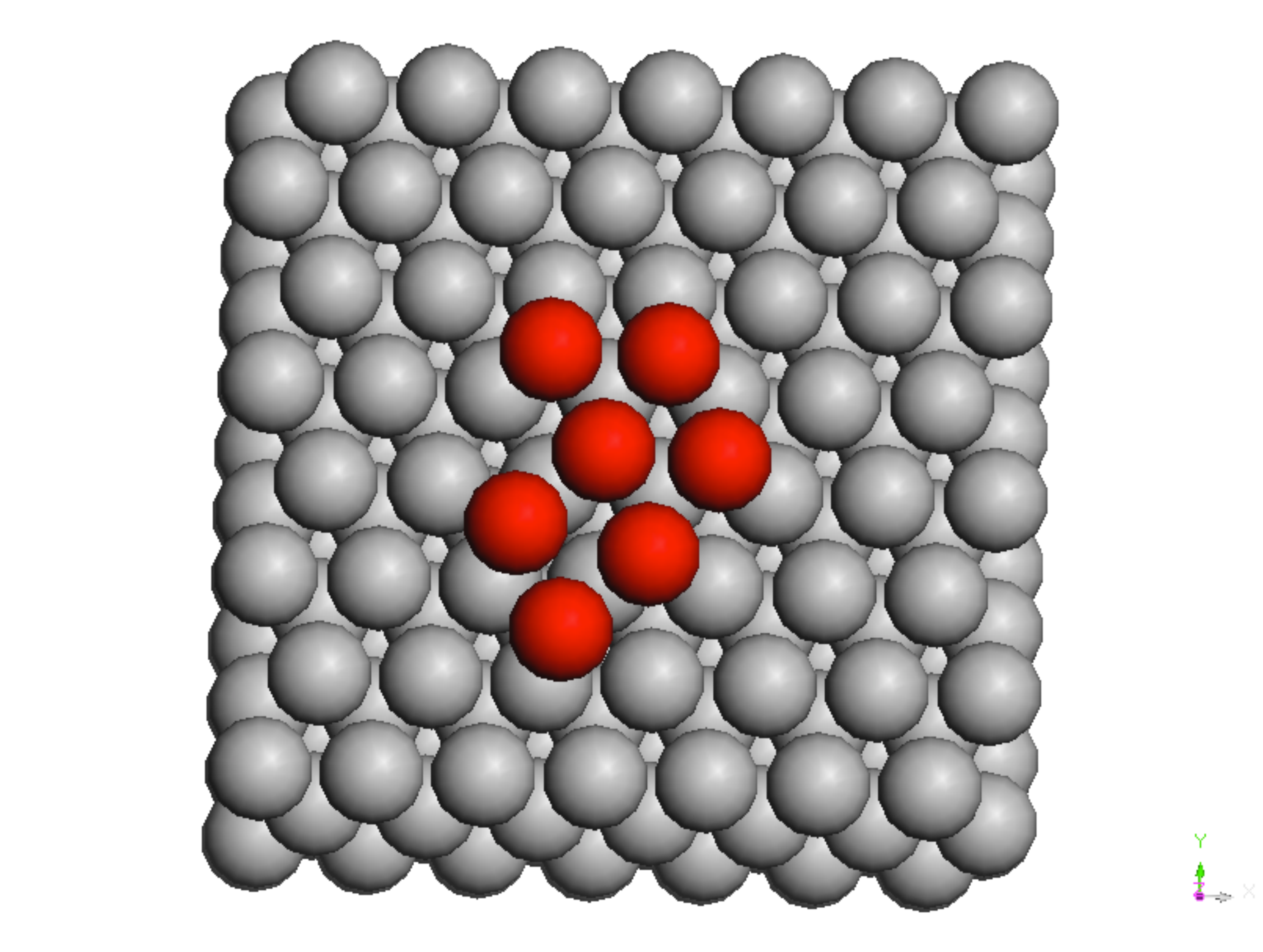}
                \caption{SP2}
                \label{fig:sp2}
        \end{subfigure}%
        
\caption{ The 7-atom island model.  Two local minima and two saddle points are shown and denoted as 
MIN1, SP1, MIN2, SP2, respectively.}\label{pic:175}
\end{figure}
We applied  our  iterative minimization method to this large-scale system.  The initial guess of position is chosen near the minima. And the initial direction  is randomly selected.  The eigenvector corresponding to the minimum eigenvalue is solved by 
an efficient  LOR method proposed  in \cite{LOR2013}.
The maximum step size in the subproblem for the position is set as 0.2. In this implementation, we used nonlinear conjugate-gradient method with
the tolerance set to $10^{-16}$ so that the subproblem is solved accurately enough.
The accuracy of  the each entry in the force at the saddle points  is between $10^{-10}$ and $10^{-11}$.
The error is defined as the Euclidean distance from the current position to the saddle.

The  numerical results  are presented in Table \ref{tab:atom}.
Since our initial guess is very close to the local minimum, 
it is not surprising that  the first several iteration steps 
have a slow decay of the errors 
since the effect of following the  eigenvector of the smallest eigenvalue  has not kicked in.
The fast convergence rate is observed  as expected in the second part.
In this example, the exact solver for each subproblem was used in Table \ref{tab:atom}, thus 
the computational  overhead  is large, compared with other algorithms requiring no subproblems to solve.
We also  introduced a simple inexact solver by limiting only two iterations of CG for subproblem. The resulting 
convergence rate   deteriorates  due to inexact solver and   linear convergence is observed.
The right  balance between the fast  convergence rate  and the large  computational overhead requires a careful design of the
tolerance in the inexact solver.

\begin{table}[htbp]
\begin{tabular}{|c|c|c|c||c|c|c|}\hline 
 Iter  & $V+W_1$      & $V+W_2$      & $\frac{2V+W_1+W_2}{2}$ & $V+W_1$      & $V+W_2$      & $\frac{2V+W_1+W_2}{2}$ \\\hline
 $1$   & 2.014e+000   & 1.832e+000   & 1.803e+000             & 1.633e+000   & 1.695e+000   & 1.521e+000 \\
 $2$   & 1.837e+000   & 1.695e+000   & 1.760e+000             & 1.599e+000   & 1.575e+000   & 1.488e+000 \\
 $3$   & 1.729e+000   & 1.575e+000   & 1.693e+000             & 1.535e+000   & 1.314e+000   & 1.433e+000 \\
 $4$   & 1.621e+000   & 1.315e+000   & 1.603e+000             & 1.446e+000   & 8.668e-001   & 1.336e+000 \\
 $5$   & 1.454e+000   & 8.668e-001   & 1.536e+000             & 1.312e+000   & 4.061e-001   & 1.167e+000 \\
 $6$   & 1.345e+000   & 4.496e-001   & 1.420e+000             & 1.114e+000   & 2.897e-001   & 9.808e-001 \\
 $7$   & 1.129e+000   & 1.605e-001   & 1.205e+000             & 9.250e-001   & 1.875e-001   & 7.974e-001 \\
 $8$   & 6.903e-001   & 3.335e-001   & 1.009e+000             & 7.405e-001   & 1.072e-001   & 6.113e-001 \\
 $9$   & 3.189e-001   & 8.653e-002   & 8.068e-001             & 5.605e-001   & 5.076e-002   & 4.407e-001 \\
 $10$  & 2.552e-001   & 9.040e-003   & 6.063e-001             & 3.855e-001   & 5.951e-003   & 2.679e-001 \\
 $11$  & 1.297e-001   & 3.398e-005   & 4.252e-001             & 2.016e-001   & 8.782e-006   & 1.058e-001 \\
 $12$  & 1.170e-002   & 6.333e-008   & 2.526e-001             & 3.005e-002   & 1.132e-007   & 1.903e-002 \\
 $13$  & 1.536e-004   & 2.641e-010   & 1.011e-001             & 5.290e-004   & 1.579e-009   & 5.277e-004 \\
 $14$  & 9.017e-008   &              & 1.141e-002             & 1.367e-008   &              & 8.758e-007 \\
 $15$  & 3.907e-010   &              & 1.487e-004             & 1.135e-009   &              & 3.347e-008 \\
 $16$  &              &              &7.792e-008              &              &              &            \\
  \hline
\end{tabular}
\caption{Errors of 6  runs with random initial guesses near the local minima as well as with different additional potentials.
The left 3 runs start from the initial guesses near MIN1 and converge to SP1.
The right 3 runs start from the initial guesses near MIN2 and converge to SP2, respectively.}
\label{tab:atom}
\end{table}

\section{Concluding remarks}

This paper presents a new formulation of iterative minimization to the saddle search problem.
In this formulation, the problem is solved by iteratively solving a sequence of 
minimization subproblems. 
At each  iteration, 
 the rotation step of determining the softest eigenvector $v$ is followed 
by a  nonlinear optimization   for the subproblem to update the $x$ variable.
We have proved the  local quadratic convergence rate of  the  new scheme.  
This new scheme is closely connected to the gentlest  ascent dynamics (GAD) and other ``eigenvector-following''
algorithms such as the dimer method. However, our subproblem is not limited only on the direction of $v$, but includes
the information of the original energy function in all directions
to update  the $x$ variable in $\Real^d$. The  quadratic convergence rate  theoretically established here   is 
 promising for further numerical improvement in practice  and indicates  that this could be the  best rate for    ``eigenvector-following''-class algorithms. 
In a forthcoming paper, we shall address the implementation  of efficient algorithms based on this formulation.
We are also interested in the 
saddle points of the free energy  landscape in collective variables, where 
the free energy function $V$ is not known, but the force, the Hessian and even the third order perturbation  can be simultaneously
computed  from one single, but expensive,  run of constrained molecular dynamics \cite{string-collective2006}.

\section*{acknowledgment}
The work of W. Gao is supported by the National Natural Science Foundation of China under grants 91330202, Shanghai Science and Technology Development Funds 13dz2260200, 13511504300.
  X. Zhou acknowledges the  financial support of  CityU Start-Up Grant (7200301) and Hong Kong Early Career Schemes (109113). The authors would like to thank the anonymous referees for valuable suggestions.
\bibliography{./gad}

\bibliographystyle{plain} 

\end{document}